\documentclass[a4paper]{scrartcl}
\usepackage{a4wide}
\usepackage{amsthm,amsmath,amsfonts,amssymb}
\usepackage{color}
\usepackage{graphicx}

\def\R{\mathbb{R}}
\newcommand{\p}{\phi}
\newcommand{\po}{{\phi^\circ}}
\newcommand{\pso}{{\psi^\circ}}
\newcommand{\Om}{\Omega}
\renewcommand{\H}{\mathcal{H}}
\def\dist{\textup{dist}}
\def\Div{\textup{div}\,}
\def\e{\varepsilon}

\theoremstyle{plain}
\newtheorem{theorem}{Theorem}[section]
\newtheorem{proposition}[theorem]{Proposition}
\newtheorem{corollary}[theorem]{Corollary}
\newtheorem{lemma}[theorem]{Lemma}
\theoremstyle{remark}
\newtheorem{remark}[theorem]{Remark}
\theoremstyle{definition}
\newtheorem{definition}[theorem]{Definition}

\newcommand{\comment}[1]{}
\newcommand{\new}[1]{#1}

\begin{document}

\title{Anisotropic and crystalline mean curvature flow of mean-convex sets}

\author{
  {Antonin Chambolle}
                \thanks{CMAP,
        Ecole Polytechnique, CNRS, Institut Polytechnique de Paris,
        Palaiseau, France,
        \newline e-mail: \texttt{antonin.chambolle@polytechnique.fr}} \ ,
        \
        {Matteo Novaga}
        \thanks{Dipartimento di Matematica, Universit\`a di Pisa,
        Largo B. Pontecorvo 5, 56127 Pisa, Italy,
         \newline e-mail: \texttt{matteo.novaga@unipi.it}}
}
\date{}
\maketitle

\begin{abstract}
\noindent We consider a variational scheme for the
anisotropic and crystalline mean curvature flow of sets with strictly positive anisotropic mean curvature. 
We show that such condition is preserved by the scheme, and
we prove the strict convergence in $BV$ of the  time-integrated perimeters of the approximating evolutions,
extending a recent result of De~Philippis and Laux to the anisotropic setting.
We also prove uniqueness of the flat flow obtained in the limit.
\end{abstract}


\medskip


\tableofcontents

\section{Introduction}

We are interested in the anisotropic mean curvature flow of sets 
with positive anisotropic mean curvature. More precisely, following \cite{CMP,CMNP19} 
we consider a family of sets $t\mapsto E(t)$ governed by the geometric evolution law 
\begin{equation}\label{eq:MCF} 
V(x,t) = -\psi(\nu_{E(t)})\, \kappa_{E(t)}^{\p}(x),
\end{equation}
where  $V(x,t)$ denotes the normal velocity of the boundary $\partial E(t)$ at $x$, $\p$ is a given
 norm
\new{or, more generally, a possibly non-symmetric convex, one-homogeneous function}
on $\R^d$, $\kappa_{E(t)}^{\p}$ is the {\em anisotropic mean curvature} of $\partial E(t)$ associated with the anisotropy $\p$,  and $\psi$ is another 
\new{convex, one-homogeneous function},
usually called mobility, evaluated at the outer unit normal $\nu_{E(t)}$ to $\partial E(t)$.
\new{Both $\p$ and $\psi$ are real-valued and positive away from $0$.}
We recall that when $\p$ is differentiable in $\R^{d}\setminus\{0\}$, then $\kappa_{E}^{\p}$ is given by
the tangential divergence of the so-called {\em Cahn-Hoffman vector field}~\cite{CahnHoffmann}
\begin{equation}\label{kappaphi}
\kappa_{E}^{\p}=\Div_{{\tau}}\left(\nabla \p(\nu_E)\right) ,
\end{equation}
while in general \eqref{kappaphi} should be replaced with the differential inclusion
\begin{equation*}
\kappa_{E}^{\p} = \Div_{{\tau}}\left(n_E^\p\right), \qquad n_E^\p \in \partial \p(\nu_E).
\end{equation*}
It is well-known that \eqref{eq:MCF} can be interpreted as gradient flow of the anisotropic perimeter
\[
P_\p(E) = \int_{\partial E}\p(\nu_E)d\H^{d-1}\,,
\]
and one can construct global-in-time weak solutions by means of the variational scheme introduced by
Almgren, Taylor and Wang~\cite{ATW93} and, independently, by  Luckhaus and Sturzenhecker~\cite{LS95}.
Such scheme consists in building  a family of tme-discrete evolutions by an iterative minimization {procedure} and in considering 
 any limit  of these  discrete evolutions, as the time step $h>0$ vanishes, as an admissible  solution to the  geometric motion, usually referred to as a {\em flat flow}. The problem which is solved at each step takes the form~\cite[\S 2.6]{ATW93} $E_h^n := T_h E_h^{n-1}$, where $T_h E$ is a solution of 
\begin{equation}\label{eq:defatw}
\min_F P_\p(F) +
 \frac{1}{h}\int_{F} d^{\pso}_E (x) dx,
\end{equation}
where $d^{\pso}_E$ is the signed distance function of $E$, with respect to the anisotropy $\pso$, which is defined as
\begin{equation}\label{eq:defdE}
  d^\pso_E(x) := \inf_{y\in E} \pso(x-y) - \inf_{y\not\in E}\pso(y-x).
\end{equation}
In \cite{ATW93} it is proved that the discrete solution $E_h(t):=E_h^{[\frac th]}$, with $\psi=1$ and $\p$ smooth, converges to a limit flat flow 
which is contained in the zero-level set of the (unique)  viscosity solution of \eqref{eq:MCF}. Such a result has been  
extended in \cite{CMP,CMNP19} to general anisotropies $\psi,\p$. In the isotropic case $\phi=\psi=|\cdot |$ it is shown in \cite{LS95}
that $E_h(t)$ converges to a distributional solution $E(t)$ of \eqref{eq:MCF}, under the assumption that the perimeter is continuous in the limit, that is,
\begin{equation}\label{eq:per}
\lim_{h\to 0}\int_0^T P(E_h(t))\, dt =\int_0^T P(E(t))
\qquad \text{for }T>0.
\end{equation}
Recently, \new{it has been shown in \cite{DPL}} that the continuity of the perimeter holds if the initial set is {\em outward minimizing} for the perimeter (see Section \ref{sec:outer}), a condition which implies the mean convexity and which is preserved by the variational scheme \eqref{eq:defatw}.

In this paper we generalize the result in  \cite{DPL} to the general anisotropic case, where the continuity of the perimeter was previously known only in the convex case \cite{BCCN06}, as a consequence of the convexity preserving property of the scheme. Such result is obtained under a stronger condition of strong outward minimality of the initial set, which is also preserved by the scheme and implies the strict positivity of the anisotropic mean curvature. 
As a corollary, we obtain the continuity of the volume and of the (anisotropic) perimeter of the limit flat flow. 

\smallskip

The plan of the paper is the following: In Section \ref{sec:def} we introduce the notion of outward minimizing set, and we recall the variational scheme proposed 
by Almgren, Taylor and Wang in~\cite{ATW93}. We also show that the scheme preserves the strict outward minimality.
In section \ref{sec:arrival} we show the strict $BV$-convergence of the discrete arrival time functions, we prove the uniqueness of the limit flow,
and we show continuity in time of volume and perimeter, and in Section \ref{sec:ex} we give some examples. 
Eventually, in Appendix \ref{sec:onesh} we recall some results on $1$-superharmonic functions, adapted to the anisotropic setting.

\section*{Acknowledgements}
The authors would like to thank the anonymous referees for their comments
and suggestions. The current version, much improved with respect to
the initial one, owes a lot to their effort.


\section{Preliminary definitions}\label{sec:def}
\subsection{Outward minimizing sets}\label{sec:outer}

\begin{definition}\label{defmean}
Let $\Omega$ be an open subset of $\R^d$ and let
$E\subset\subset \Om$ be a finite perimeter set.
We say that $E$ is  outward minimizing in $\Omega$ if
\begin{equation*}
  P_\p(E) \le P_\p(F) \quad \forall F\supset E, F\subset\subset \Om.\eqno (MC)
\end{equation*}
\end{definition}
\noindent Note that, if $E$, $\p$ are regular, $(MC)$ implies that the $\p$-mean curvature of $\partial E$ is non-negative.

We observe that such a set 
satisfies the following density bound: \new{There exists $\gamma>0$ such that, for all points $x\in E$ satisfying
$|B(x,\rho)\setminus E|>0$ for all $\rho>0$, it holds:
\begin{equation}\label{eq:lwrdensity}
  \frac{|B(x,\rho)\setminus E|}{|B(x,\rho)|} \ge \gamma ,
\end{equation}
whenever $B(x,\rho)\subset \Om$. As a consequence, whenever $x\in E$ is a point of Lebesgue density $1$,
there exists $\rho>0$ small enough such that
$|B(x,\rho)\setminus E|=0$. 
Therefore, identifying the set $E$ with its
points of density $1$, we always assume (unless otherwise explicitly stated)
that $E$ is an open subset of $\R^d$.}

Conversely if $E\subset\R^d$ is bounded and $C^2$, $\p$ is $C^2(\R^d\setminus\{0\})$, and its mean curvature is positive, then one
can find $\Om\supset\supset E$ such that $E$ is outward minimizing in $\Om$. More precisely, if $E$
is of class $C^2$ then, in a neighborhood of $\partial E$, $d^\po_E$ is $C^2$, while in a smaller
neighborhood we even have \new{$\Div\nabla \p(\nabla d^\po_E)\ge \delta$}, for some $\delta>0$. Let $\Om$ be the union
of $E$ and this neighborhood, and set \new{$n_E^\p:=\nabla \p(\nabla d^\po_E)$}: then if $E\subset F\subset\subset \Om$,
\[
  P_\p(F)\ge \int_{\partial^* F}n^\p_E\cdot\nu_F d\H^{d-1}=
  -\int_\Om n^\p_E\cdot D\chi_F
\]
while by construction $P_\p(E)=-\int_\Om n^\p_E\cdot D\chi_E$.
Hence,
\[
  P_\p(F) \ge P_\p(E) - \int_\Om n^\p_E \cdot D(\chi_F-\chi_E)
  = P_\p(E) + \int_{F\setminus E}\Div n^\p_E
  \ge P_\p(E)+\delta |F\setminus E|.
\]

\comment{
In general, if we consider a set $E$ such that
\begin{equation*}
  \lim_{s\to 0} \frac{1}{2s} \{ x\in\R^d:|d_E^\p(x)|\le s \} = P_\p(E), \eqno (Mink)
\end{equation*}
and if there is a neighborhood of $\partial E$ (not necessarily smooth) and
a vector field $n^\p_E\in\partial \p(\nabla d^\po_E)$ with $\Div n^\p_E\ge 0$,
then again, if $\Om$ is the union of $E$ and this neigborhood, $E$ is variational
mean convex in $\Om$. If in addition $\Div n^\p_E\ge \delta>0$ near $\partial E$, we find
moreover that
\begin{equation*}
  P_\p(E) \le P_\p(F) - \delta |F\setminus E| \quad \forall F\supset E, F\subset\subset \Om. 
\end{equation*}
}

Observe (see~\cite[Lemma 2.5]{DPL}) that equivalently, one can express this as:
\begin{equation*}
  P_\p(E\cap F) \le P_\p(F) - \delta |F\setminus E| \quad \forall F\subset\subset \Om.\eqno (MC_\delta)
\end{equation*}
Clearly, condition $(MC_\delta)$ is stronger and reduces to $(MC)$ whenever $\delta=0$.

\begin{remark}[Non-symmetric distances] \label{rem:nsd}
As in the standard case (that is when $\pso$ is smooth and even), the signed ``distance''
  function defined in~\eqref{eq:defdE} is easily seen to satisfy the usual properties of a signed
  distance function. First, it is Lipschitz continuous, hence differentiable almost everywhere.
  Then, if $x$ is a point of
differentiability, $d^\pso_E(x)>0$ and $y\in\partial E$ is such that $\pso(x-y)=d^\pso_E(x)$, then  for $s>0$ small
and $h\in\R^d$ $d^\pso_E(x+sh) \ge \pso(x+sh-y)\ge \pso(x-y)+s z\cdot h$ for any $z\in\partial\pso(x-y)$
and one deduces that $\partial \pso(x-y)=\{\nabla d^\pso_E(x)\}$. If $d^\pso_E(x)<0$, one writes that
$\pso(y-x)=-d^\pso_E(x)$ for some $y\in \partial E$ and uses $\pso(y-x-sh)\ge \pso(y-x) -s z\cdot h$ for
some $z\in\partial\pso(y-x)$, hence $d^\pso_E(x+sh)-d^p_E(x)\le s z\cdot h$ to deduce
now that $\partial \pso(y-x) = \{\nabla d^\pso_E(x)\}$. In all cases, one has \new{$\psi(\nabla d^\pso_E(x))=1$} a.e.~in $\{d^\pso_E\neq 0\}$
(while of course $\nabla d^\pso_E(x)=0$ a.e.~in $\{d^\pso_E=0\}$),
and \new{$\nabla d^\pso_E(x)\cdot (x-y) = d^\pso_E(x)$}, which shows that $y\in x-d^\pso_E(x)\partial \p(\nabla d^\pso_E(x))$).
\end{remark}

\subsection{The discrete scheme}\label{sec:scheme}

We now consider here the discrete scheme introduced in~\cite{LS95,ATW93} and its
generalization in~\cite{CaCha,BCCN06,CMP,CMNP18}. It is based on the following
process: given $h>0$, and $E$ a (bounded) finite perimeter set,
we define $T_hE$ as a minimizer of 
\[
  \min_F P_\p(F) + \frac{1}{h}\int_F d_E^\pso(x)dx \eqno (ATW)
\]
where $d^\pso_E$ is defined in~\eqref{eq:defdE}. 
If $E\subset\subset\Omega$ satisfies $(MC)$ in $\Om$, it is clear that for $h>0$ small enough, one has $T_hE\subset E$.
Indeed, for $h$ small enough one has $\overline{T_h E}\subset\Om$, and it follows from $(MC)$ \new{(more precisely, in the form $(MC_\delta)$ for
  $\delta=0$)}
that
\begin{equation}\label{eq:ThMC} \new{
  P_\p(T_h E \cap E) + \frac{1}{h}\int_{T_h E \cap E}  d^\pso_E(x) dx
 \le
  P_\p(T_h E) + \frac{1}{h}\int_{T_h E}  d^\pso_E(x) dx -
  \frac{1}{h}\int_{T_h E\setminus E}  d^\pso_E(x) dx ,
}
\end{equation}
which implies that $|T_h E\setminus E|=0$.
We recall in addition that in this case, $T_hE$ is also $\p$-mean convex in $\Om$, see the proof of~\cite[Lemma 2.7]{DPL}.
If $E$ satisfies $(MC_\delta)$ in $\Omega$ for some $\delta>0$, we can improve the inclusion $T_hE\subset E$.

\begin{lemma}\label{lem:dist}
  Assume that $E\subset\subset\Omega$ satisfies $(MC_\delta)$ \new{in $\Omega$, for some $\delta>0$}. Then
  for $h>0$ small enough, it holds
  \new{$$T_hE + \{ \pso\le \delta h\}\subset E.$$}
  In particular, $d^\pso _{T_h E}\ge d^\pso_E + \delta h$  and $T_h E   \subset \{d^\pso_E\le -\delta h\}$.
\end{lemma}

\begin{proof}
  Let $h>0$ small enough so that $T_h E\subset E$ and $E+\{\pso\le \delta h\}\subset\Omega$.
  Choose $\tau$ with $\pso(\tau)<\delta h$ and consider $F:=T_hE+\tau$. We show that also $F\subset E$.
  The set $F\subset\subset \Omega$ is  a minimizer of
  \[
    P_\p(F) + \frac{1}{h} \int_F d^\pso_E(x-\tau)dx.
  \]
  In particular, we have
  \begin{multline*}
    P_\p(F) + \frac{1}{h} \int_F d^\pso_E(x-\tau)dx
    \le
    P_\p(F\cap E) + \frac{1}{h} \int_{F\cap E} d^\pso_E(x-\tau)dx
    \\    \le P_\p(F)+
    \frac{1}{h} \int_F d^\pso_E(x-\tau)dx
    -\int_{F\setminus E} \frac{1}{h}d^\pso_E(x-\tau) +\delta\, dx.
  \end{multline*}
  By definition of the signed distance function, for $x\not\in E$,
  $d^\pso_E(x-\tau)\ge -\pso(x-(x-\tau))=-\pso(\tau)>-\delta h$ so that
  if $|F\setminus E|>0$ we have a contradiction. 
  We deduce that $T_h E + \{\pso\le \delta h\}\subset E$.

  In particular, if $x\in T_hE$ and $y\not \in E$ is such that $d^\pso_{E}(x) = -\pso(y-x)$,
  then $y'=y-\delta h (y-x)/\pso(y-x)\not\in T_hE$ hence $d^\pso_{T_h E}\ge -\pso(y'-x) = d^\pso_E(x)-\delta h$.
  If $x\in E\setminus  T_h E$, $d^\pso_E(x) = -\psi(y-x)$ for some $y\in\overline{\Omega\setminus E}$, and $d_{T_h E}^\pso(x) = \psi(x-y')$
  for some $y'\in T_h E$. Since $\psi(x-y')+\psi(y-x)\ge \psi(y-y') \ge \delta h$ we conclude. Eventually if $x\not\in E$,
  for $y\in T_h E$ with $d^\pso_{T_h E}(x)=\po(x-y)$ we have $y+\delta h(x-y)/\po(x-y)\in E$, so that
  $d^\pso_E(x)\le \po(x-y)-\delta h =d^\pso_{T_h E}(x)-\delta h$. This shows that $d^\pso_{T_h E} \ge d^\pso_E + \delta h$.
\end{proof}

\begin{corollary}\label{thm:separatesets}
  Under the assumptions of Lemma \ref{lem:dist}, for any $n\ge 1$, we have $T_h^{n+1} E +\{\pso \le \delta h\}\subset T_h^{n} E$ and  $d^\pso_{T_h^n E}\ge d^\pso_E \new{+}%
  \delta n h.$
\end{corollary}

\begin{proof}
  The first statement is obvious by induction: Assuming that for $\tau$ with $\pso(\tau)\le\delta h$ one has $T_h^{n}E+\tau\subset T_h^{n-1} E$
  which is true for $n=1$, applying $T_h$ again and using the translational invariance we get that $T_h^{n+1}E+\tau\subset T_h^n E$.
  The second statement is obviously deduced. 
  \new{Indeed we can reproduce the end of the previous proof
    to find that $d^\pso_{T_h^n E} \ge d^\pso_{T_h^{n-1} E} + \delta h$,
    the conclusion follows by induction.}
\end{proof}

\begin{remark}[Density estimates]\label{rema}
  There exists $\gamma>0$, depending only on $\phi$ and the dimension, and
  $r_0>0$, depending also on $\psi$, such that the following holds: for $x$ such that  $|B(x,r)\cap T_hE|>0$ for all $r>0$ one
  has $|B(x,r)\cap T_hE|\ge\gamma r^d$ if $r<r_0h$. For the complement, as $T_h E$ is $\p$-mean convex in
  $\Omega$, we have as before that for $x$ such that $|B(x,r)\setminus T_h E|>0$ for all $r>0$,
  one has $|B(x,r)\setminus T_h E|\ge\gamma r^d$ for all $r$ with $B(x,r)\subset \Omega$, \textit{cf}~\eqref{eq:lwrdensity}.
\end{remark}

\subsection{Preservation of the outward minimality}

In the sequel, we show some further properties of the discrete evolutions and their limit.
An interesting result  in~\cite{DPL}
is that the $(MC_\delta)$-condition is preserved during the evolution.
We prove that it is also the case in the anisotropic setting.

We first show the following result: 
\begin{lemma}\label{lem:isop}
\new{Let $\delta>0$ be such that there exists a set
  $E\subset\subset\Omega$ satisfying $(MC_\delta)$ in $\Om$. Then $\delta|F|\le P_\p(F)$ for any $F\subset\subset\Om$,
  that is, the empty set also satisfies $(MC_\delta)$ in $\Om$.}
\end{lemma}

\begin{proof}
  By  $(MC_\delta)$ we have $\delta |F|= \delta |F\cap E|+\delta |F\setminus E|
  \le \delta |F\cap E| + (P_\p(F)-P_\p(F\cap E))$, so that
  it is enough to show the result for $F\subset E$.
  For $s>0$, we let $E_s$ be the largest minimizer of
  \begin{equation}\label{eq:ATWs}
    P_\p(E_s)+\frac{1}{s}\int_{E_s} d^\pso_{E} \,dx,
  \end{equation}
  which is obtained as the level set $\{w_s\le 0\}$ of the (Lipschitz continuous) solution $w_s$ of the equation
  \begin{equation}\label{eq:Eulers}
    -s\Div z_s + w_s = d^\pso_{\new{E}},\quad  z_s\in \partial \p(\nabla w_s),
  \end{equation}
  see for instance~\cite{CaCha,AlChNo} for details.
\new{A standard translation argument shows that the function $w_s$ satisfies $\psi(\nabla w_s)\le \psi(\nabla d_E^\pso)= 1$ a.e.~in $\R^d$.}
  We also let $E'_s:=\{w_s<0\}$ be the smallest minimizer of~\eqref{eq:ATWs}. \new{By construction, the set $E_s$ is closed while $E'_s$ is open.}
  
\new{ By Lemma \ref{lem:dist} it follows
 that there exists $s_0>0$ such that $E_s\subset\subset E$ for all $s<s_0$. Moreover, 
 being $E$ an open set, we also have $|E_s\Delta E|\to 0$ as $s\to 0$. 
 Indeed, given $x,\rho$ with $B(x,\rho)\subset E$, by comparison we have that 
   $x\in E_s$ for all $s< c \rho^2$, where $c>0$ depends only on $d,\p$ and $\pso$.}
  

Since $P_\p(E_s)\le P_\p(E)$, by the lower semicontinuity of $P_\p$
we get that $\lim_{s\to 0} P_\p(E_s)=P_\p(E).$ 
 We also claim that 
  \begin{equation}\label{eqF}
  \lim_{s\to 0} P_\p(F\cap E_s)=P_\p(F).  
  \end{equation}
Indeed, it holds
  \[
     P_\p(F\cup E_s)+P_\p(F\cap E_s)  \le P_\p(E_s) + P_\p(F),
   \]
   and $|E\setminus (F\cup E_s)|\to 0$ as $s\to 0$, so that
   \[
     P_\p(E)+\limsup_{s\to 0} P_\p(F\cap E_s) 
     \le \limsup_{s\to 0} \left( P_\p(F\cup E_s) + P_\p(F\cap E_s) \right)
     \le P_\p(E)+P_\p(F),
   \]
   which shows the claim.
   
   Again by Lemma~\ref{lem:dist} we know that $d_E^\pso\le -s\delta$ on $\partial E_s=\{w_s\le 0\}$.
   If $x\in E_s$ and $y\in \partial E_s$, $w_s(x)\ge w_s(y)-\pso(y-x)=-\pso(y-x)$ (using \new{$\psi(\nabla w_s)\le 1$}).
   If $z\not\in E$ and $y\in[x,z]\cap \partial E_s$, by one-homogeneity of $\pso$ we get
   one has $\pso(z-x) = \pso(z-y)+\pso(y-x)$,
   so that $0\le w_s(x)+\pso(y-x) = w_s(x) + \pso(z-x) - \pso(z-y) \le w_s(x) + \pso(z-x) - s\delta$. Taking
   the infimum over $z$, we see that $s\delta\le w_s(x)-d^\pso_E(x)$. Hence $\Div z_s \ge \delta$ a.e.~in \new{$E_s$}, so that
 \begin{equation}\label{eqfine}
   P_\p(F\cap E_s)\ge \int_\Om \Div z_s \chi_{F\cap E_s} \ge \delta |F\cap E_s|.
   \end{equation}
   The thesis now follows recalling \eqref{eqF} and letting $s\to 0$ in \eqref{eqfine}. 
\end{proof}

\new{
\begin{remark}
Notice that the constant $\delta$ in Lemma  \ref{lem:isop} is necessarily bounded above by the anisotropic Cheeger constant of $\Omega$
 (see \cite{CaFaMe}) defined as
\[
h_\p(\Omega) := \inf_{F\subset\subset\Om, F\ne\emptyset} \ \frac{P_\p(F)}{|F|}\,.
\]
\end{remark}
}

 We can now deduce the following:
 
\begin{lemma}\label{lem:ATWdelta}
  Let $\delta>0$, $E\subset\subset\Omega$ satisfy $(MC_\delta)$ in $\Om$, $h>0$ small enough, 
  and let $T_hE\subset E$ be the solution of $(ATW)$. Then
  $T_hE$ also satisfies $(MC_\delta)$ in $\Om$.
\end{lemma}

\begin{proof}
  We remark that the sets $E_s$, $E'_s$ defined in the proof of Lemma~\ref{lem:isop}
  satisfy $E_s\subset E'_{s'}$ for $s>s'$. This follows from the
  fact that the term $s\mapsto d^\pso_E(x)/s<0$ is increasing for $x\in E$.
  \new{As a consequence 
    $E_s\setminus E'_s=\partial E_s=\partial E'_s$ and is Lebesgue negligible, for all $s$ but a countable number.}
  Also,
  if $s_n\to s$, $s_n < s$, then $E_{s_n}\to E_s$, while if $s_n>s$, $\Om\setminus E'_{s_n}$ converges
  to $\Om\setminus E'_s$. Moreover, as the sets satisfy uniform density estimates (for $n$ large enough), these
  convergences are
  also in the Hausdorff sense. In particular, we deduce that $E\setminus E'_s=\bigcup_{0<s'\le s} (E_{s'}\setminus E'_{s'})$
  (we recall $E_{s'}\setminus E'_{s'}=\{w_{s'}=0\}$).

  Let $\e>0$.
  \new{From the proof of Lemma~\ref{lem:isop}, if $h$ small enough so that Lemma~\ref{lem:dist} is valid,
    we know that $\Div z_s \ge \delta$ a.e.~in $E_s$.
    In addition, since  $w_s$ in~\eqref{eq:Eulers} satisfies $\psi(\nabla w_s)\le \psi(\nabla d^\pso_E)=1$
    a.e., then
    $\Div z_s$ is $(C/s)$-Lipschitz for a constant $C$ depending only on $\psi$.
  We deduce that }
there exists $\eta>0$ (depending only on \new{$\e,\psi$}) such that for any $s\in \new{(0,h)}$, 
in $N_s = \{x:\dist(x,E_s)< s\eta\}$, one has $\Div z_s\ge \delta-\e$\new{.}

  Let $h>\bar s>\underline{s}>0$, \new{with $\bar s$ and $\underline{s}$ chosen so that
  $\partial E'_{\bar s}=\partial E_{\bar s}$ and $\partial E'_{\underline{s}}=\partial E_{\underline{s}}$} . The set
  $E_{\underline{s}}\setminus E'_{\bar s}$ is covered by the open sets
  \new{
    $\tilde{N}_s = \{ x: 0<\dist(x,E'_s) \le
    \dist(x,E_s)<\underline{s}\eta/2\} \subset N_s$,
    $\underline{s}/2<s<h$.  Indeed, for
    $x\in E_{\underline{s}}\setminus E'_{\bar s}$, either
    $x\in E_s\setminus\overline{E'}_s\subset \tilde{N}_s$ for some
    $s\in[\underline{s},\bar{s}]$, or $x$ is approached by points in
    $x_n\in E_{s_n}$, $s_n\downarrow s$, so that
    $\dist(x,E_{s_n})<\underline{s}\eta/2$ for $n$ large enough and
    $x\in \tilde{N}_{s_n}$.}

  Hence one can extract a finite covering indexed by $s_1> s_2 >\dots> s_{N-1}$. We observe that
  necessarily, $h>s_1>\bar s$ and we let $s_N:=\underline{s}$. \new{In addition, for $1\le i\le N-1$ one must have
    $\partial E'_{s_{i+1}}\subset \tilde{N}_{s_i}$. Indeed, $\partial E'_{s_{i+1}}\cap \tilde{N}_{s_j}=\emptyset$
    for $j\ge i+1$, while if $x\in\partial E'_{s_{i+1}}\cap \tilde{N}_{s_j}$ for some $j< i$, since $\partial E_{s_i}$
    is in between $\partial E_{s_j}$ and $\partial E'_{s_{i+1}}$ one also has $x\in\tilde {N}_{s_i}$. In fact,
  we deduce $E'_{s_{i+1}}\setminus\overline{E'}_{s_i}\subset \tilde{N}_{s_i}$}
  
  Let $F\subset\subset\Om$ and up to an infinitesimal translation, assume $\H^{d-1}(\partial^*F\cap \partial \new{E'_{s_i}})=0$ for $i=1,\dots, N$.
  One has for $i\in \{ 1,\dots,N\}$,
  \begin{multline*} P_\p(\new{E'_{s_{i+1}}}\cap F)-P_\p(\new{E'_{s_i}}\cap F)
    \new{\ = \int_{\partial^*(E'_{s_{i+1}}\cap F)\setminus \overline{E'}_{s_i}}\p(\nu_{E'_{s_i}\cap F})d\H^{d-1} -
    \int_{F\cap \partial E'_{s_i}} \p(\nu_{E'_{s_i}})d\H^{d-1}}
  \\  \ge \int_{\partial^* [F\cap \new{E'_{s_{i+1}}}\!\setminus \new{E'_{s_i}}]}
  z_{s_i}\cdot\nu_{[F\cap E'_{s_{i+1}}\!\setminus E'_{s_i}]}
    \,d\H^{d-1}
    = \int_{F\cap \new{E'_{s_{i+1}}}\!\setminus \new{E'_{s_i}}}\Div z_{s_i} dx
    \ge (\delta-\e) |F\cap \new{E'_{s_{i+1}}}\!\setminus \new{E'_{s_i}}|.
  \end{multline*}
  \new{In the first inequality, we have used that $z_{s_i}\in \partial\p(\nu_{E'_{s_i}})$ so
    that $z_{s_i}\cdot \nu_{E'_{s_i}} = \p(\nu_{E_{s_i}})$ a.e.~on $\partial E'_{s_i}$ (and
    $z_{s_i}\cdot\nu\le\p(\nu)$ for all $\nu$),
    while in the last inequality, we have used $\Div z_{s_i}\ge \delta-\e$ in $\tilde{N}_{s_i}$.}
  \new{Hence,} summing from $i=1$ to $N$, we find that \new{(recalling that $E'_{\underline{s}}=E_{\underline{s}}$ up
    to a negligible set)}
  \[ P_\p(\new{E'_{s_1}}\cap F)\le P_\p(E_{\underline s}\cap F) -
    (\delta-\e) |(E_{\underline s}\setminus \new{E'_{s_1}})\cap F|.\]
  Since $E_{\underline s}$ is outward minimizing, $P_\p(E_{\underline s}\cap F)\le P_\p(E\cap F) \le P_\p(F)-(\delta-\e) |F\setminus E|$, so that:
  \[
    P_\p(\new{E'_{s_1}}\cap F)\le P_\p(F)-(\delta-\e) (|F\setminus E|+|(E_{\underline s}\setminus \new{E'_{s_1}})\cap F|).
  \]
  Sending $\bar s<s_1$ to $h$ and $\underline s$ to $0$, we deduce that
  $P_\p(E_h\cap F)\le P_\p(F)-(\delta-\e)|F\setminus E_h|$ hence the thesis holds,
  since $\e$ is arbitrary.
\end{proof}

\new{
\begin{remark}\label{rem:hsmall}
  Let us observe that both in Lemma~\ref{lem:dist} and in Lemma~\ref{lem:ATWdelta},
  as well as in Corollary~\ref{thm:separatesets},
  the conclusion holds as soon $h$ is small enough to have $\overline{T_hE}\subset\Om$
  (since in this case~\eqref{eq:ThMC} holds and $T_h E\subset E$),
  and $E+\{\pso\le \delta h\}\subset\Om$.
  In particular,  in all these
  results if $E'\subset E$ is another set satisfying
  $(MC_\delta)$ and $h$ is small enough for $E$, then it is also small enough for $E'$.
\end{remark}
}

\section{The arrival time function}\label{sec:arrival}

Consider an open set $\Omega\subset \R^d$ and a set $E^0\subset\subset \Omega$ such that $(MC_\delta)$ holds for some $\delta>0$.
As usual~\cite{LS95,ATW93} we let $E_h(t) := T_h^{[t/h]}(E^0)$, here $[\cdot]$ denotes the integer part.
Being the sets $T_h^n (E^0)$ mean-convex, we can choose an open representative.
We can define the \textit{discrete arrival time function} as
\[
u_h(x):= \max \{ t\chi_{E_h(t)}(x) , t\ge 0\},
\]
which is a l.s.c.~function\footnote{We can say that $u_h$ is a function in $BV(\Om)$ with compact support
and such that its approximate lower limit~ $u_h^-$ is lower semicontinuous.}
which, thanks to the co-area formula, satisfies
\begin{equation}\label{eq:uhonsh}
  \int_\Om \p(-Du_h) \le \int_{\overline{\Om}} \p(-Dv)
\end{equation}
for any $v\in BV(\R^d)$ with $v\ge u_h$ and $v=0$ in $\R^d\setminus\Omega$. 
\new{In particular,  $u_h$ is ($\p$-)$1$-superharmonic in the sense of~Definition \ref{defA}.}
 One can easily see that $(u_h)_h$ is uniformly bounded in $BV(\Om)$ so that a subsequence $u_{h_k}$
 converges in $L^1(\Om)$ to some $u$, which again is ($\p$-)$1$-superharmonic.

 In addition, since $E^0$ satisfies $(MC_\delta)$,
  thanks to Corollary~\ref{thm:separatesets} we have that $u_h$ satisfies a global Lipschitz bound. More
 precisely, for $x,y\in\Omega$ there holds
 \[
 u_h(x)-u_h(y) \le h+ \frac{\po(y-x)}{\delta}.
 \]
 Indeed, one has $u_h(x)=t\Rightarrow u_h(x+\tau)\ge t-h$ for any $t\ge 0$ and $\tau$
 with $\po(\tau)\le \delta h$. The claim follows by induction. 
 
 As a consequence we obtain that $u_{h}$ converges uniformly, up to a subsequence,  to a limit function $u$, which is also 
Lipschitz continuous, and satisfies 
 \begin{equation}\label{ulip}
 u(x)-u(y) \le \frac{\po(y-x)}{\delta}
 \end{equation}
 for any $x,y\in\Om$.  Moreover, recalling Lemma \ref{lem:ATWdelta}, we have that the functions $u_h$ and $u$ are 
$(\phi,\delta)$-$1$-superharmonic, in the sense of Definition \ref{defA} below.

We now show that the function $u$ is unique, and is the
 arrival time function of the anisotropic curvature flow starting form $E^0$, in the sense of~\cite{CMNP19}.
\new{ In particular, there is no need to pass to a subsequence for the convergence of $u_h$ to $u$ in the argument above.}
 
\begin{theorem}\label{th:limit}
  \new{Under the previous assumption on $E^0$,
  the arrival time function $u_h$ converge, as $h\to 0$,  to a unique limit $u$ such that
  $t\mapsto \{ u\le t\}$ is a solution of~$\eqref{eq:MCF}$ starting from $E^0$.
  Moreover it holds $$\lim_{h\to 0} \int_\Omega\p(-Du_h) = \int_\Omega \p(-Du)\,.$$}
\end{theorem}

 \begin{proof}
   For $s>0$ we let $E^s:=\{u>s\}$. Notice that, since $E^0$ is open, as in the proof of Lemma \ref{lem:isop} 
   we have $\bigcup_{s>0}E^s = E^0$.
   
   As a consequence of the existence and uniqueness result in~\cite{CMP,CMNP19}, 
   for a.e.~$s>0$ the arrival time functions $u^s_h\le u_h$ of the discrete flows $T^{[t/h]}_h E^s$ 
   converge uniformly to a unique limit $u^s$. In particular, considering
   the subsequence $u_{h_k}$, one has $u^s\le u$. On the other hand, thanks to Corollary~\ref{thm:separatesets}
   and the Remark \ref{rema}, given $s>0$ there is $\tau_s>0$ such that $T^{[\tau_s/h]}_h E^0\subset E^s$.
   Then, $T^{[\tau_s/h]+n}_h E^0\subset T^n_h E^s$ by induction so that $u_h-\tau_s-h \le u^s_h$. If $v$ is the
   limit of a converging subsequence of $(u_h)$, we deduce $v-\tau_s \le u^s\le u$.
   Sending $s\to 0$ we deduce $v\le u$.
   Since this is true for any pair $(u,v)$ of limits of converging subsequences of $(u_h)$, this limit
   is unique and $u_h\to u$.

 The last statement is already proved in~\cite{DPL} in a simple way: One just needs to show
  that $$\limsup_h \int_\Omega\p(-Du_h) \le \int_\Omega \p(-Du)\,.$$ 
  Since $(u_h)_h$
  converges uniformly to $u$, given $\e>0$, one has $u_h\le u+\e$ for $h$ small enough. On the other hand,
  since all these functions vanish out of $E^0$, it follows $u_h\le u+\e\chi_{E^0}$. Hence,
  being $u_h$  $\p$-$1$-superharmonic,
  \[
    \int_\Omega \p(-Du_h) \le \int_\Omega \p(-D(u+\e\chi_{E^0}))=\int_\Omega \p(-Du) + \e P_\p(E^0)
  \]
  for $h$ small enough,  and the thesis follows. 
\end{proof}

Theorem  \ref{th:limit} shows that the scheme starting from a strict $\p$-mean convex set always
converges to a unique flow, with no loss of anisotropic perimeter.
In particular,
\new{in dimension $d\le 3$ and if $\p$ is smooth and elliptic},
following~\cite{LS95} one can show
that the limit satisfies a distributional formulation of the anisotropic curvature flow.
\new{
More precisely, we say that a couple of functions $(X,v)$, with
\[
X:\Om\times [0,+\infty)\to\{0,1\}\in L^\infty(0,+\infty;BV(\Omega)),\quad 
v:\Om\times [0,+\infty)\to\R\in L^1(0,+\infty;L^1(\Omega,|DX(t)|)),
\]
is a \emph{$BV$-solution} to \eqref{eq:MCF} with initial datum $E^0$
 if the following holds:
For all $T>0$, 
$\zeta \in C^\infty(\overline\Om\times [0,T];\R^d)$ with $\zeta|_{\partial\Om\times [0,T]=0}$, 
and $\xi \in C^\infty(\overline\Om\times [0,T])$ with $\xi|_{\partial\Om\times [0,T]=0}$ and $\xi(T)=0$, we have
\begin{eqnarray} \label{eq04}
  && \int_0^T\left[\int_{\Om} \left( \text{div}\zeta + \nabla\p\left(-\frac{DX(t)}{|DX(t)|}\right)\,\nabla\zeta\,\frac{DX(t)}{|DX(t)|}\right)\p(-DX(t))
     +  v \zeta\cdot DX(t) \right]dt= 0,
\\ \label{eq05}
&&  \int_0^T\int_{\Om} X\, \partial_t\xi\,dxdt + \int_{E^0} \xi(x,0)\,dx = -   \int_0^T\int_{\Om} v\, \xi\, \psi(-DX(t))dt.
\end{eqnarray}

Reasoning as in~\cite[Theorem 2.3]{LS95}  one can prove the following, for $\p$ $C^{2,\alpha}$ and elliptic:
\begin{theorem}\label{th:LS}
Let $d\le 3$, let $u$ be the limit function in Theorem \ref{th:limit}, and let $X(x,t) := \chi_{\{u>t\}}(x)$.
Then there exists $v\in L^1(0,+\infty;L^1(\Omega,|DX(t)|))$ such that the couple $(X,v)$ is a \emph{$BV$-solution} to \eqref{eq:MCF}.
\end{theorem}

\begin{proof}
We only explain the adaptions to~\cite{LS95} required to prove this result.
Most of the proof remains unchanged, as it relies on estimates (such as basic density estimates)
which remain valid in the new setting.
However some difficulties arise in Section~2 of \cite{LS95} and in particular in the proof of Proposition~2.2,
which uses the regularity theory for minimal surfaces. Indeed, one first should assume that the dimension $d\le 3$,
$\p$ is elliptic and $C^{2,\alpha}$ for some $\alpha>0$,
in order to benefit from the regularity theory for anisotropic integrands (see\cite{ASS, Simon}) and be able to use the Bernstein argument at the end of page 265 of \cite{LS95}. This allows to show~\eqref{eq05},
which is a small variant of~\cite[Eq.~(0.5)]{LS95} (here $f=0$) whith the signed distance function replaced with the $\pso$-signed distance function.

In order to show~\eqref{eq04}, the Euler-Lagrange equation \cite[Eq. (0.7)]{LS95} has to be modified, with the curvature term on the left hand side replaced by the first variation of $P_\p$, which can be found in \cite[Ex.~20.7]{Maggi}.

\end{proof}
}

\begin{remark}[Continuity of volume and perimeter]
As is well-known for general flat flows (see~\cite{LS95,Caraballo}), the limit motion
$t\mapsto \{u> t\}$ is $1/2$-H\"older in $L^1(\Omega)$, in the sense that, for $s>t>0$,
\begin{equation}\label{contvol}
  |\{s>u\ge t\}\cap \Omega|\le C|t-s|^{1/2},
\end{equation}
where $C$ depends on the dimension and on the perimeter of the initial set. In particular,
$|\{u=t\}|=0$ for all $t>0$, so that up to a negligible set, $\{u>t\}=\{u\ge t\}$.
For $t=0$ it may happen that \new{$|\partial{\{u>0\}}|>0$}, as shown in the second example below.
A direct consequence of \eqref{contvol} is the absence of fattening for the evolution of an 
outward minimizing set.

In addition, since \new{the sets $\{u>t\}$ satisfy $(MC_\delta)$} for $t>0$, 
for $s>t \ge 0$ we have that
\[
  P_\p(\{u>s\}) + \delta |\{s\ge u >  t\}| = P_\p(\{u> t\}),
\]
so that $t\mapsto P_\p(\{u>t\})$ is strictly decreasing until extinction.
Since $\bigcup_{s>t} \{u>s\}=\{u>t\}$ we also get that $t\mapsto P_\p(\{u>t\})$ is
right-continuous.  
\new{Whether this function could jump or not remains an open question in this generality,
however the continuity has been proven in \cite{MS} in the classical isotropic case $\p(\cdot)=\psi(\cdot)=|\cdot|$.
}
\end{remark}


\section{Examples}\label{sec:ex}

\subsection{The case $\delta=0$}

If the initial datum $E^0$ satisfies only $(MC)$ we shall consider two cases: 
If $\p$ and $\psi$ are smooth and elliptic and $\partial E^0$ is smooth,
then there exists a smooth solution to \eqref{eq:MCF} on a time interval $[0,\tau)$, for some $\tau>0$ (see \cite[Chapter 8]{Lunardi}). 
Then, by the parabolic maximum principle, the solution $E(t)$ becomes strictly mean-convex for $t\in (0,\tau)$. In particular,
for any $\e\in (0,\tau)$ there exist $\delta_\e>0$ and an open set $\Om_\e$ such that 
$E(t_\e)\subset\subset\Om_\e$, $\delta_\e\to 0$ as $\e\to 0$,
 and $E(t)$ satisfies $(MC_{\delta_\e})$ in $\Om_\e$ for $t\in (\e,\tau)$.
As a consequence, the previous results hold in all the time intervals $[\e,+\infty)$, \new{so that
the limit} function $u$ is unique and continuous, and it is locally Lipschitz continuous in the interior of $E^0$.

On the other hand, for an arbitrary anisotropy $\p$, the function $u$ could be discontinuous on the boundary of $E^0$.
As an example in two dimensions, we take $\psi(\xi,\eta) = \p(\xi,\eta)=|\xi|+|\eta|$
and the cross-shaped initial datum 
\[
E^0 := \left([-1,1]\times [-2,2]\right)\cup\left([-2,2]\times [-1,1]\right)  \subset\R^2\,.
\]
It is easy to check that $E^0$ is outward minimizing, so that $E(t)\subset E^0$ is also outward minimizing
for all $t>0$. Moreover, the solution $E(t)=\{(x,y):\,u(x,y)\ge t\}$ 
is unique (see for instance \cite{Gi-Gi:01}) and can be explicitly described as follows (see Figure 1):
\begin{equation}\label{eq:Et}
E(t) =  \left\{\begin{array}{ll}
\left([-1,1]\times [-2+t,2-t]\right)\cup\left([-2+t,2-t]\times [-1,1]\right)  & \text{for } t\in [0,1],
\\[2mm]
\new{\big[-\sqrt{1-2(t-1)},\sqrt{1-2(t-1)}\,\big] \times
\big[-\sqrt{1-2(t-1)},\sqrt{1-2(t-1)}\,\big]} & \text{for } t\in \left[1,3/2\right],
\\[2mm]
\emptyset & \text{for }t> 3/2.
\end{array}\right.
\end{equation}
In particular, 
the function $u\in BV(\R^2)$  is 
\textit{discontinuous} on $\partial E^0\setminus \partial ([-2,2]\times [-2,2])$.

\begin{figure}[h]
\centering
\includegraphics[width=0.4\textwidth,trim=7.2cm 3cm 6.5cm 2cm,clip]{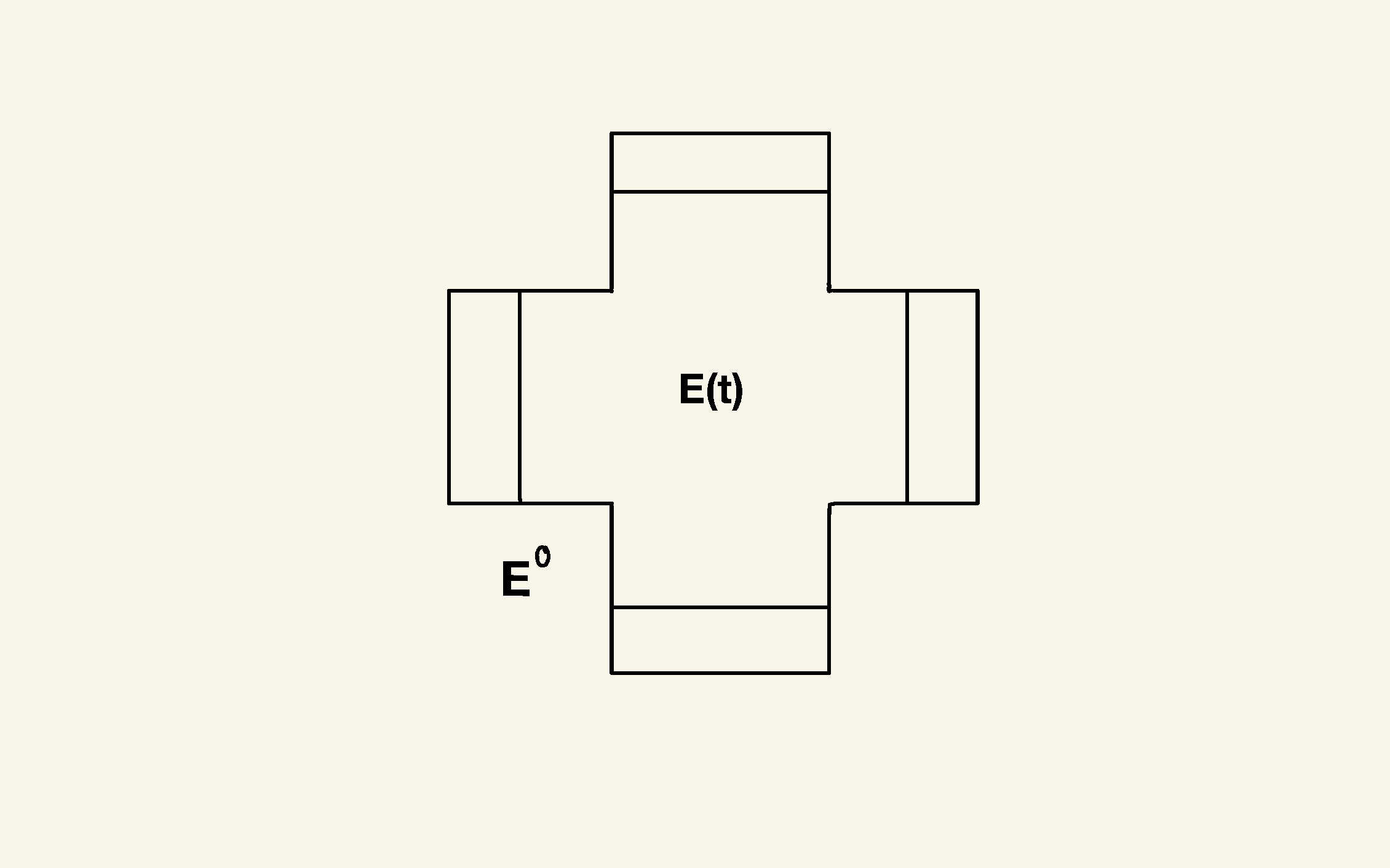}
\caption{The evolving set $E(t)$.}
\end{figure}

We observe that Formula~\eqref{eq:Et}
for $E(t)$ can be easily obtained by finding explicit solutions to~$(ATW)$, starting from
$E_L = ([-1,1]\times [-L,L])\cup([-L,L]\times [-1,1])$, $L>1$.
A ``calibration'' is given by the following vector field $z$, defined in $E_L$:
\[
  z(x,y) = \begin{cases} (x,y) & \textup{ if } |x|\le 1, |y|\le 1,\\
    (x, \pm 1) & \textup{ if } |x|\le 1, 1\le \pm y\le L,\\
    (\pm 1, y) & \textup{ if } 1\le \pm x\le L, |y|\le 1.
  \end{cases}
\]
One has $\Div z = 1+\chi_{[-1,1]^2}$ in $E_L$, $z(x,y) \in\{\pso\le 1\}$, and 
$P_\p(E_{\ell}) = \int_{\partial E_\ell} z\cdot\nu\,d\H^1$ for any $1\le\ell\le L$. Hence, if $L-h\ge 1$ and $F\subset E_L$, we have
\begin{multline*}
  P_\p(F)  + \int_{F} \frac{d^\pso_{E_L}}{h} dx \ge \int_{\partial F} \nu\cdot z d\H^1 + \int_{F} \frac{d^\pso_{E_L}}{h} dx
  \\ = \int_{\partial F} \nu\cdot z d\H^1 - \int_{\partial E_{L-h}} \nu\cdot zd\H^1 + P_\p(E_{L-h}) 
  + \int_{F} \frac{d^\pso_{E_L}}{h} dx
  \\ = \int z\cdot(D\chi_{E_{L-h}} - D\chi_F)  + P_\p(E_{L-h}) +
  \int_{E_{L-h}} \frac{d^\pso_{E_L}}{h} dx +   \int_{E_L} (\chi_F-\chi_{E_{L-h}}) \frac{d^\pso_{E_L}}{h} dx
\\  =   P_\p(E_{L-h}) +  \int_{E_{L-h}} \frac{d^\pso_{E_L}}{h} dx
  + \int_{E_L} (\chi_F-\chi_{E_{L-h}})\left(\frac{d^\pso_{E_L}}{h}  + 1 + \chi_{[-1,1]^2}\right)dx.
\end{multline*}
Now, the last integral is nonnegative, since $d^\pso_{E_L}/h+1\le 0$ in $E_{L-h}$, and is positive outside.
As a consequence, $E_{L-h}$ solves $(ATW)$ for $E=E_L$, and
one deduces the first line in~\eqref{eq:Et}. The proof of the second line in~\eqref{eq:Et}
is a standard computation (see for instance \cite{BCCN06}).

\subsection{Continuity of the volume up to $t=0$}

We provide, \new{in dimension $d=2$}, an example of an open set $E$ satisfying $(MC_\delta)$ for some $\delta>0$,
and such that \new{$|\partial{\{u>0\}}|>0$}.
\new{The set is built as a countable union of disjoint disks.}

Let $(x_n)_{n\ge 1}$  be a dense sequence of rational points in $\Omega:=B(0,1)\subset\R^2$.
We shall construct inductively a sequence $(r_n)_{n\ge 1}$ of 
positive numbers with \new{$\sum_n r_n<+\infty$} such that the following property holds:
Letting $E_0=\emptyset$ and $E_n = E_{n-1}\cup B(x_n,r_n)$ for $n\ge 1$, the sets  
$E_n$ all satisfy $(MC_{\delta})$ in $\Omega$ for some $\delta>0$.

Notice first that there exists $\delta>0$ such that each ball $B(x,r)\subset\Omega$
satisfies $(MC_{2\delta})$ in $\Omega$.
Choose now $r_1>0$ in such a way that $E_1=B(x_1,r_1)\subset\Omega$,
then $E_1$ satisfies $(MC_{2\delta})$. Assume now by induction 
that $E_{n}$ satisfies $(MC_{(1+1/n)\delta})$. Then, if $d_n:=\dist(x_{n+1},E_{n})=0$ we 
let $r_{n+1}=0$, so that $E_{n+1}=E_n$. Otherwise, if $d_n>0$ we choose $r_{n+1}\in (0,2^{-n})$ 
in such a way that 
\begin{equation}\label{condrn}
r_{n+1} \le  \min\left(\frac 12 \left( \frac 1n-\frac{1}{n+1}\right) \frac{\delta d_n^2}{2\pi C},\frac{d_n}{6}\right),
\end{equation}
where the constant $C>0$ will be chosen later in \textit{Case 3.}
Let also $\mathcal{N}\subset\mathbb N$ be the (infinite) set of indices such that $r_n>0$.

\new{Assuming that $E_n$ satisfies $(MC_{\delta+\delta/n})$, which is true
  for $n=1$,}
Let us check that $E_{n+1}$ satisfies $(MC_{\delta(1+\delta/(n+1)})$. We consider a set $F$ of finite perimeter
such that $E_{n+1}\subset F\subset\Om$, and we distinguish three cases:

\noindent\textit{Case 1.} $|F\cap B(x_{\new{n+1}},d_n)|\ge d_n^2/C$. In this case we have
\begin{eqnarray*}
P(F)&\ge& P(E_n) + \left( 1+\frac 1n\right)\delta |F\setminus E_n| 
\\ 
&\ge&
P(E_{n+1}) - 2\pi r_{n+1} + \left( 1+\frac{1}{n+1}\right)\delta |F\setminus E_{n}|
+  \left( \frac 1n-\frac{1}{n+1}\right) \delta |F\cap B(x_{\new{n+1}},d_n)|
\\ &\ge&
P(E_{n+1}) + \left( 1+\frac{1}{n+1}\right)\delta |F\setminus E_{n+1}| 
+ \left( \frac 1n-\frac{1}{n+1}\right) \frac{\delta d_n^2}{C} - 2\pi r_{n+1}
\\ &\ge&
P(E_{n+1}) + \left( 1+\frac{1}{n+1}\right)\delta |F\setminus E_{n+1}| ,
\end{eqnarray*}
where in the last inequality we used \eqref{condrn}.

\noindent\textit{Case 2.} $|F\cap B(x_{\new{n+1}},d_n)|\le d_n^2/C$ and $\H^1(F\cap \partial B(x_{\new{n+1}},r))=0$
for some $r\in (r_{n+1},d_n)$. 
In this case, we write $F=F_1\cup F_2$, with 
$F_1=F\cap B(x_{\new{n+1}},r)\supset B(x_{\new{n+1}},r_{n+1})$ 
and $F_2=F\setminus B(x_{\new{n+1}},r)\supset E_n$, and we have
\begin{eqnarray*}
P(F_1)&\ge& P(B(x_{\new{n+1}},r_{n+1})) + 2\delta  |F_1\setminus B(x_{\new{n+1}},r_{n+1})|
\\
P(F_2)&\ge& P(E_n)+\left( 1+\frac 1n\right)\delta |F_2\setminus E_n|.
\end{eqnarray*}
Summing up the two inequalities above, we get
\begin{eqnarray*}
P(F)&=&P(F_1)+P(F_2)\ \ge\ P(E_{n+1}) 
+\left( 1+\frac 1n\right)\delta 
\left(|F_1\setminus B(x_{\new{n+1}},r_{n+1})| + |F_2\setminus E_n|\right)
\\
&=&P(E_{n+1}) +\left( 1+\frac 1n\right)\delta |F\setminus E_{n+1}|.
\end{eqnarray*}

\noindent\textit{Case 3.} $|F\cap B(x_{\new{n+1}},d_n)|\le d_n^2/C$ and $\H^1(F\cap \partial B(x_{\new{n+1}},r))>0$
for a.e.~$r\in (r_{n+1},d_n)$. In this case, by co-area formula we have
\[
  \new{
 \int_{\frac{d_n}{6}}^{\frac{d_n}{3}}\H^1(F\cap \partial B(x_{\new{n+1}},r))\,dr
 =
 \left|F\cap \left(B\left(x_{\new{n+1}},\frac{d_n}{3}\right)\setminus B\left(x_{\new{n+1}},\frac{d_n}{6}\right)\right)\right| \le \frac{d_n^2}{C}.
}
\]
It follows that there exists $\rho_1\in (d_n/6,d_n/3)$ such that 
\[
\H^1(F\cap \partial B(x_{\new{n+1}},\rho_1))\le \frac{6d_n}{C}.
\]
Similarly we have
\[
  \new{
    \int_{\frac{2d_n}{3}}^{d_n} \H^1(F\cap \partial B(x_{\new{n+1}},r))\,dr
    =   \left|F\cap \left(B(x_{\new{n+1}},d_n)\setminus B\left(x_{\new{n+1}},\frac{2d_n}{3}\right)\right)\right| 
\le \frac{d_n^2}{C},
}
\]
and there exists $\rho_2\in (2d_n/3,d_n)$ such that 
\[
\H^1(F\cap \partial B(x_{\new{n+1}},\rho_2))\le \frac{3d_n}{C}.
\]
Using that
$\H^1(F\cap \partial B(x_{\new{n+1}},r))>0$ for all $r\in (r_{n+1},d_n)$ we deduce that
\begin{itemize}
\item 
either for a.e.~$r\in (\rho_1,\rho_2)$, \new{it holds $\H^0(\partial^*F\cap \partial B(x_{\new{n+1}},r))\ge 2$,}
 and it follows
that $P(F,B(x_{\new{n+1}},\rho_2)\setminus B(x_{\new{n+1}},\rho_1))\ge 2(\rho_2-\rho_1)\ge 2d_n/3$,
\item 
or for a set of positive measure of radii $r\in (\rho_1,\rho_2)$ one has $\H^1(F\cap \partial B(x_{\new{n+1}},r))=2\pi r$.
\new{In this case, observe that for a.e.~$y\in \partial B(x_{n+1},\rho_1)\setminus F$, the ray from $x_{n+1}$ to $\partial B(x_{n+1},r)$
  through $y$ crosses $\partial^* F$ at least once outside of $\overline{B}(x_{n+1},\rho_1)$ so that the projection of $\partial^*F\cap B(x_{\new{n+1}},\rho_2)\setminus B(x_{\new{n+1}},\rho_1)$ onto $\partial B(x_{n+1},\rho_1)$ has
  measure at least $2\pi\rho_1-6d_n/C$. Hence,
}
\[
  P(F,B(x_{\new{n+1}},\rho_2)\setminus B(x_{\new{n+1}},\rho_1))\ge 2\pi \rho_1 - 6d_n/C\\
  \ge d_n(\pi/3-6/C)\ge 2d_n/3
\]
provided we have chosen $C\ge 18/(\pi-2)$.
\end{itemize}

Then, proceeding as in the previous case we let $F_1=F\cap B(x_{\new{n+1}},\rho_{1})$
and $F_2=F\setminus B(x_{\new{n+1}},\rho_{2})$, and we have
\begin{align*}
P(F)=&\ P(F_1)+P(F_2) -  \H^1(F\cap \partial B(x_{\new{n+1}},\rho_1)) 
       -\H^1(F\cap \partial B(x_{\new{n+1}},\rho_2))
  \\ & \hspace{6.2cm} + \ P(F,B(x_{\new{n+1}},\rho_2)\setminus B(x_{\new{n+1}},\rho_1))
\\
\ge &\ P(E_{n+1}) 
+\left( 1+\frac 1n\right)\delta 
\left(|F_1\setminus B(x_{\new{n+1}},r_{n+1})| + |F_2\setminus E_n|\right)-\frac{9d_n}{C}+\frac{2d_n}{3}
\\
\ge &\ P(E_{n+1}) +\left( 1+\frac 1n\right)\delta |F\setminus E_{n+1}| - 
\left( 1+\frac 1n\right)\delta\frac{d_n^2}{C} -\frac{9d_n}{C}+\frac{2d_n}{3}
\\
\ge &\ P(E_{n+1}) +\left( 1+\frac 1n\right)\delta |F\setminus E_{n+1}| 
-\frac{2\delta + 9}{C} d_n+\frac{2d_n}{3}
\\
    \ge &\ P(E_{n+1}) +\left( 1+\frac 1n\right)\delta |F\setminus E_{n+1}|,
\end{align*}
as long as we choose $C\ge 3(2\delta + 9)/2$.

We proved that $E_n$ satisfies $(MC_{\delta})$ for all \new{$n\in \mathcal N$, therefore also the limit set
$$E:=\bigcup_{n\in \mathcal N} E_n=\bigcup_{n\in \mathcal N} B(x_n,r_n)$$} satisfies
$(MC_\delta)$ in $\Om$. 
In this case, the solution $u$ in Theorem  \ref{th:limit} is explicit and it is given by \new{
\[  
u(x)=  \sum_{n\in \mathcal N} \frac{(r_n^2 - |x-x_n|^2)^+}{\new{2}}.
\]}
Notice that we have
\[
\partial \{u>0\} = \partial E = \overline{B(0,1)}\setminus E,
\]
so that $|\partial \{u>0\}|=\pi-|E|>0$.


\appendix
\section{$\mathbf{1}$-superharmonic functions}\label{sec:onesh}

The goal of this appendix is to recall
some results proved in \cite{SchevenSchmidt1} on $1$-superharmonic functions,
to give precise statements in the anisotropic case,
and to propose some simple proofs, when possible.

\begin{definition}\label{defA}
We say that $u$ is ($\p$-)$1$-superharmonic in $\Omega$ if
$\{u\neq 0\}\subset\subset\Omega$ and for any $v$ with
$v\ge u$, $\{v\neq 0\}\subset\subset\Omega$, one has
\[
  \int_\Om \p(-Du) \le \int_\Om \p(-Dv),
\]
or, equivalently, for any $v$ with compact support in $\Om$,
\[
  \int_\Om \p(-D(u\wedge v)) \le \int_\Om \p(-Dv). \eqno (SH)
\]
Given $\delta>0$, we say that  $u$ is ($(\p,\delta)$-)$1$-superharmonic in $\Omega$ if
$\{u\neq 0\}\subset\subset\Omega$ and one has:
\[
  \int_\Om \p(-D(u\wedge v)) \le \int_\Om \p(-Dv) - \delta\int_\Om (v-u)^+dx
\quad \forall \ v, \{v\neq 0\}\subset\subset\Om. \eqno (SH_\delta)
\]
Equivalently, $u$ is a minimizer of $$\int_\Om \p(-Du) - \delta\int_\Om u dx,$$ 
with respect to larger competitors with the same boundary condition.
\end{definition}
Obviously then, $u\ge 0$ (using $v=u^+$ in $(SH)$). 
Notice that $\chi_E$ is $1$-superharmonic if and only if the set $E$  is outward minimizing.

Observe that, in this case, the set $E^0=\{u>0\}$ has
finite perimeter and satisfies $(MC_\delta)$.
Indeed, for $E\subset F\subset\subset\Om$, letting $v=\e\chi_F$ for $\e>0$,
we have
\begin{multline*}
  \int_\Om \p(-D(u\wedge \e\chi_F))  = \int_0^\e P_\p(\{u>s\}\cap F) ds
  \\  \le \e P_\p(F)-\delta\int_\Om (\e\chi_F-u)^+dx
  = \e\left(P_\p(F)-\delta\int_\Om (\chi_F-u/\e)^+dx\right).
\end{multline*}
Hence:
\[
 \int_0^1 P_\p(\{u>t\e\}\cap F) dt
\le P_\p(F)-\delta\int_\Om (\chi_F-u/\e)^+dx .
\]
Sending $\e\to 0$, we deduce $(MC_\delta)$.

In particular, it follows from Lemma~\ref{lem:isop} that for any $v\in BV(\Om)$ compactly supported,
$\delta \int_\Om |v| dx\le \int_\Om\p(-Dv)$.
We then deduce that if $u$ satisfies $(SH_\delta)$ also $u\wedge T$ for any $T>0$.
Indeed,
\[
  \int_\Om \p(-D((u\wedge T) \wedge v)) \le \int_\Om \p(-D(v\wedge T)) - \delta\int_\Om ((v\wedge T)-u)^+dx
\]
On the other hand,
\[
  \int_\Om \p(-D(v\wedge T))=\int_\Om\p(-Dv) -\int_\Om\p(-D(v-T)^+)
  \le\int_\Om \p(-Dv)-\delta\int_\Om (v-T)^+\,dx,
\] 
and it follows
\[
  \int_\Om \p(-D((u\wedge T) \wedge v)) \le \int_\Om \p(-Dv) - \delta\int_\Om (v-(u\wedge T))^+dx.
\]

Then, the following characterization holds:
\begin{proposition} Let $u$ satisfy $(SH_\delta)$. Then
  there exists $z\in L^\infty(\Omega;\{\po\le 1\})$ with
  $\Div z\ge \delta$,
  $[z,Du^+]=|Du|$ in the sense of measures (equivalently,
  $\int_\Om u^+ \Div z \,dx = \int \p(-Du)$),
  and $\Div z=\delta$ on $\{u=0\}$.
\end{proposition}

\begin{corollary}
 Let $u$ satisfy $(SH_\delta)$. Then for any $s> 0$, $\{u^+\ge s\}$ and $\{u^+>s\}$ satisfy $(MC_\delta)$. 
\end{corollary}
Here, $u^+$ is as usual the superior approximate limit of $u$ (defined $\H^{d-1}$-a.e.) and $[z,Du^+]$
the pairing in the sense of Anzellotti~\cite{Anzellotti}.

\begin{proof}
  For $n\ge 1$, let $v_n$ be the unique minimizer of
  \begin{equation}\label{eq:rofn}
    \min_{v=0\ \partial\Om} \int_\Om \p(-Dv) + \int_\Om\frac{n}{2}(v-u\wedge n)^2 -\delta v \,dx.
  \end{equation}
  (the boundary condition is to be intended in a relaxed sense,
  adding a term $\int_{\partial \Om} |\textup{Tr}v|\p(\nu_\Om)d\H^{d-1}$ in the energy
  if the trace of $v$ on the boundary does not vanish).
  The Euler-Lagrange equation for this problem asserts the existence of a field $z_n\in L^\infty(\Om;\{\po\le 1\})$
  with bounded divergence such that
  \[
    \Div z_n  + n v_n = n(u\wedge n) + \delta 
  \]
  a.e.~in $\Om$, and $\int_\Om \Div z_n v_n \,dx = \int_\Om \p(-Dv_n)$. On the other hand
  $\int_\Om\p(-Dv_n)\le\int_\Om\p(-D(u\wedge n))\le\int_\Om\p(-Du)$ and we have $v_n\to u$,
  $\int_\Om\p(-Dv_n)\to\int_\Om\p(-Du)$ as $n\to\infty$.

  We show that $v_n\le u\wedge n$. Indeed,
  $\int_\Om \p(-D (v_n\wedge u\wedge n))\le \int_\Om \p(-Dv_n)- \delta\int_\Om (v_n-(u\wedge n))^+ dx$,
  while
  $\int_\Om (v_n-(u\wedge n))^2 dx \ge \int_\Om ((v_n\wedge u\wedge n)-(u\wedge n))^2$.
  Hence,
  \begin{multline*}
    \int_\Om \p(-D (v_n\wedge u\wedge n)) +
    \frac{n}{2}\int_\Om ((v_n\wedge u\wedge n)-(u\wedge n))^2 - \delta\int_\Om (v_n\wedge u\wedge n)dx
\\    \le
    \int_\Om \p(-Dv_n) + \frac{n}{2}\int_\Om (v_n-(u\wedge n))^2 dx - \delta\int_\Om v_n dx
    \\+ \delta\int_\Om (v_n-(v_n\wedge u\wedge n))     -  (v_n-(u\wedge n))^+ dx
    \\ =
        \int_\Om \p(-Dv_n) + \frac{n}{2}\int_\Om (v_n-(u\wedge n))^2 dx - \delta\int_\Om v_n dx
  \end{multline*}
  and as the minimizer $v_n$ of~\eqref{eq:rofn} is unique, we deduce $v_n=v_n\wedge u\wedge n$.
  In particular, it follows $\Div z_n\ge \delta$. (Observe that since $v_n\ge 0$, one also
  has $\Div z_n\le \delta + n(u\wedge n)$, in particular $\Div z_n=\delta$ a.e.~in $\{u=0\}$.
  Also, $\int_{\{u>0\}} \Div z_n\le P_\p(E^0)$, hence $(\Div z_n)_{n\ge 1}$ are
  uniformly bounded Radon measures. Hence, up to a subsequence, we may assume that
  $z_n\stackrel{*}{\rightharpoonup} z$ weakly-$*$ in $L^\infty(\Om;\{\po\le 1\})$ while
  $\Div z_n\stackrel{*}{\rightharpoonup} \Div z$ weakly-$*$ in $\mathcal{M}^1(\Om;\R_+)$, that is,
  as positive measures.

We now write
\[
  \int_\Om \p(-Dv_n) = \int_\Om v_n\Div z_n \, dx
  \le \int_\Om (u \wedge n)\Div z_n \,dx = \int_0^n \int_{\{u\ge s\}}\Div z_n\,dx ds,
\] 
hence, since $v_n\to u$,
\[
  \int_\Om\p(-Du)\le \limsup_{n\to\infty}\int_0^n \int_{\{u\ge s\}}\Div z_n\,dx ds
  \le \int_0^\infty \left(\limsup_{n\to\infty}\int_{\{u\ge s\}}\Div z_n\,dx\right) ds
\] 
thanks to Fatou's lemma (and the fact $\int_{\{u\ge s\}}\Div z_n\,dx\le P_\p(E^0)$
are uniformly bounded).

We now study the limit of $\int_{\{u\ge s\}}\Div z_n\,dx$, for $s>0$ given,
assuming $\{u>s\}$ has finite perimeter (this is true for a.e.~$s$, and in fact one
could independently check that $s\mapsto P_\p(\{u\ge s\})$ is nonincreasing).

We consider a set $F=\{u\ge s\}$ with finite perimeter, and we recall
$D\chi_F$ is supported on the reduced boundary $\partial^* F$.
By inner regularity, given $\e>0$, we find a compact set $K\subset \partial^* F$
with $|D\chi_F|(\Om\setminus K)<\e$. We observe that $\H^{d-1}$-a.e.~on $K$ (which
is countably rectifiable), $\chi_F$ has an upper an lower trace, respectively $\chi_F^+=1$
and $\chi_F^-=0$. By the Meyers-Serrin Theorem (or its $BV$ version, \textit{cf}~\cite{Anz-Gia-78} or \cite[Theorem~3.9]{AFP}),
there exists $\varphi_k$ a sequence of functions in $C^\infty(\Om\setminus K;[0,1])$
with $\varphi_k\to\chi_F$ and
\[
  \int_0^1 \H^{d-1}(\{x\in\Om\setminus K: \varphi_k(x)= k\})
  = \int_{\Om\setminus K}|\nabla \varphi_k|dx\to
  |D\chi_F|(\Om\setminus K)<\e.
\]
Moreover, by construction the traces of $\varphi_k$
in $K$ coincide with the traces of $\chi_F$ (see~\cite[Section 3.8]{AFP}). 

We choose for each $k$ $s_k\in [1/4,3/4]$ such that
$\H^{d-1}(\partial \{\varphi_k\ge s_k\}\setminus K)\le 2\e$.
We then define the closed (compact) sets $F_k:=\{\varphi_k \ge s_k\}\cup K$.
One has $\int_\Om |D\chi_F-D\chi_{F_k}|=\int_{\Om\setminus K}|D\chi_F-D\chi_{F_k}|
\le 3\e$. (This shows that $F$ can be approximated strongly in $BV$ norm by
closed sets.)

Then, one has $\limsup_{n}\int_{F_k}\Div z_ndx\le \int_{F_k}\Div z$ as the
measures are nonnegative and $\chi_{F_k}$ is scs.
On the other hand, $|\int_\Om \Div z_n (\chi_F-\chi_{F_k}) dx|
\le 3\e$, so that
\[
  \limsup_{n\to\infty} \int_{F}\Div z_ndx
  \le 3\e + \int_{F}\Div z + \int (\chi_{F_k}-\chi_F)\Div z
  \le 3\e + \int_{F}\Div z + \int (\chi_{F_k}-\chi_F)^+\Div z.
\] 
Notice that it is important to specify
precisely the set $F$ that we consider in the last inequality:
We pick for $F$ the complement $F^+$ of its points of density zero,
equivalently $F^+ = \{u^+\ge s\}$. In that case, up to a set
of zero $\H^{d-1}$-measure, $\chi_G:=(\chi_{F_k}-\chi_{F^+})^+=\chi_{F_k\setminus F^+}$ vanishes on $K$
pointwise, moreover at $\H^{d-1}$-a.e.~$x\in K$, $G$ has Lebesgue density $0$.
Hence $G$ coincides $\H^{d-1}$-a.e.~with a Caccioppoli set
strictly inside $\Om$ and with $\int_\Om |D\chi_G|\le 3\e$.
Thanks to~\cite[Thm~5.12.4]{Ziemer} it follows $\Div z(G)\le C\e$
for $C$ depending only on $\p$ and the dimension (see also~\cite[Prop.~3.5]{SchevenSchmidt1}).
As a consequence, since $\e>0$ is arbitrary,
\[
  \limsup_{n\to\infty}\int_{\{u\ge s\}} \Div z_n dx \le \int_{\{u^+\ge s\}} \Div z.
\]
We obtain that
\[
  \int_\Om\p(-Du)\le \int_\Om u^+\Div z.
\]
The reverse inequality also holds thanks to~\cite[Prop.~3.5, (3.9)]{SchevenSchmidt1}, and
can be proved by localizing and smoothing with kernels depending on the local
orientation of the jump. We also deduce that, for a.e.~$s>0$,
$$
\int_{\{u^+\ge s\}}\Div z =P_\p(\{u\ge s\})\,. 
$$
Note that
$s\mapsto\Div z({\{u^+\ge s\}})$ is left-continuous, and $s\mapsto\Div z({\{u^+> s\}})$ is right-continuous,
whereas $s\mapsto P_\p(\{u^+\ge s\})$ is left-semicontinuous, which implies the thesis.
\end{proof}


\end{document}